\documentclass[12pt]{amsart}
   \usepackage{amsmath,amsthm}
   \usepackage{amsfonts}   % if you want the fonts
   \usepackage{amssymb}    % if you want extra symbols
   \usepackage{url}

\advance \topmargin by -\headheight
\advance \topmargin by -\headsep
      
\evensidemargin \oddsidemargin
\usepackage[all]{xy}
\newtheorem{thm}{Theorem}

\newtheorem{lem}{Lemma}
\newtheorem{cor}{Corollary}
\newtheorem{conj}{Conjecture}

\usepackage{amssymb}

\theoremstyle{remark}
\newtheorem{rem}{Remark}

\theoremstyle{definition}

\title{Symplectic 4-manifolds with fixed point free circle actions}
\author{Jonathan Bowden}
\address{Mathematisches Institut, Universit\"at Augsburg, Universit\"atsstr.\ 14, 86159 Augsburg, Germany}
\curraddr{Max-Planck-Institut f\"ur Mathematik, Vivatsgasse 7, 53111 Bonn, Germany}
\email{jonathan.bowden@math.uni-augsburg.de}
\date{\today}

\begin{document}
\maketitle
\begin{abstract}
We show that recent results of Friedl-Vidussi and Chen imply that a symplectic 4-manifold admits a fixed point free circle action if and only if it admits a symplectic structure that is invariant under the action and we give a complete description of the symplectic cone in this case. This then completes the topological characterisation of symplectic 4-manifolds that admit non-trivial circle actions.
\end{abstract}

\section{Introduction}
Recently Friedl and Vidussi, \cite{FVi1} solved the long standing Taubes Conjecture, which classifies which 4-manifolds of the form $M \times S^1$ admit symplectic forms. Moreover, they determined exactly which cohomology classes can be represented by symplectic forms. Using recent results of D. Wise, \cite{Wi} they have extended their results to the case of non-trivial $S^1$-bundles in \cite{FVi3}. In this note we observe that their results as well of those of Chen, who obtained partial results in the fixed point free case in \cite{Chen}, imply the analogue of (\cite{FVi3}, Theorem 1.3) for all fixed point free circle actions.

Before stating our main result we fix some notation and terminology. Let $X \stackrel{p} \longrightarrow M$ be an orientable 4-manifold with a fixed point free circle action and quotient space $M = X/S^1$. The quotient space is an orbifold whose underlying topological space $|M|$ is a manifold since all the stabilisers of the $S^1$-action are necessarily cyclic and the singular locus consists of a collection of branching circles (cf.\ \cite{Bal}, \cite{Fin}). 

We let $M_{reg}$ denote the complement of an open tubular neighbourhood of the singular locus of $M$ and $X_{reg} = p^{-1}(M_{reg})$, which is an honest $S^1$-bundle so that the pushforward map $p_*$ is well-defined for cohomology classes in $H^*(X_{reg},\mathbb{R})$. The manifold $M_{reg}$ has toroidal boundary and thus one may define the Thurston norm on $H^1(M_{reg},\mathbb{R})$ in the usual fashion. Finally for $\psi \in H^2(X,\mathbb{R})$ we let $\psi_{reg}$ denote the restriction of $\psi$ to $X_{reg}$.
%A rational class in $H^1(M_{reg},\mathbb{Q})$ is then fibered it has a multiple which is dual to the fiber of a fibration over $S^1$ that is transverse to the boundary.
%Moreover, each boundary component $\partial_iM_{reg}$ has a distinguished homology class $m_i \in H_1(\partial_iM_{reg},\mathbb{Z})$ given by a meridian of the solid torus bounding $\partial_iM_{reg}$ in $|M|$. 
%If $\phi_{reg}$ is the restriction of a class $\phi \in H^1(|M|, \mathbb{Q})$, then $\phi_{reg}$ is meridianially non-degenerate if and only if 
% and that the singular locus is a fibered link.
\begin{thm}\label{class_eff}
Let $X \stackrel{p} \longrightarrow M$ be an oriented manifold admitting a fixed point free $S^1$-action with quotient space $M$ and let $\psi \in H^2(X,\mathbb{R})$. Then the following are equivalent:
\begin{enumerate}
\item $\psi$ can be represented by a symplectic form,
\item $\psi$ can be represented by an $S^1$-invariant symplectic form,
\item $\psi^2 >0$ and $p_*\psi_{reg} \in H^1(M_{reg},\mathbb{R})$ lies in the open cone over a fibered face of the Thurston norm ball and is the restriction of a class in $H^1(|M|,\mathbb{R})$.
\end{enumerate}
\end{thm}
Note that if a class $\phi \in H^1(|M|,\mathbb{R})$ is integral, then $\phi$ can be represented by a fibration over $S^1$ that is transverse to the singular locus of $|M|$ if and only if its restriction to $M_{reg}$, which we denote by $\phi_{reg}$, is fibered. Recall that a fibration of a manifold with boundary is required to be transverse to the boundary. Furthermore, since $\phi_{reg}$ is the restriction of a class in $H^1(|M|,\mathbb{R})$, it automatically vanishes on the meridian classes in $\partial M_{reg}$ so that if $\phi_{reg}$ is fibered, then the induced fibration on the boundary is necessarily meridional. Thus the fibration dual to $\phi_{reg}$ extends to $|M|$ in the desired way by filling in discs near the singular locus. In particular, part (3) of Theorem \ref{class_eff} implies that the underlying manifold $|M|$ is fibered and we obtain a positive answer to the following conjecture, which implies (\cite{Chen}, Conjecture 1.7) as a special case.
\begin{conj}[Generalised Taubes conjecture]\label{Taubes_gen}
Let $X$ be a symplectic 4-manifold that admits a non-trivial fixed point free circle action with quotient orbifold $M$. Then the (possibly empty) singular locus $L$ of $M$ is a meridionally fibered link.
% and in particular $|M|$ is fibered.
\end{conj}
Furthermore, as noted in (\cite{Chen}, p.\ 6), Theorem \ref{class_eff} completes the characterisation of which symplectic manifolds admit non-trivial $S^1$-actions. For Baldridge, \cite{Bal} showed that if a non-trivial $S^1$-action on a symplectic 4-manifold has fixed points then $X$ is rational or ruled and thus admits an $S^1$-invariant symplectic form for \emph{some} non-trivial $S^1$-action. In view of this we obtain the following corollary.
\begin{cor}
Let $X$ be a symplectic 4-manifold that admits a non-trivial $S^1$-action. Then either the action is fixed point free and the quotient space fibers over $S^1$ or $X$ is rational or ruled. In either case, $X$ admits a non-trivial symplectic $S^1$-action.
\end{cor}
%It should be noted that for fixed point free circle actions the correct analogue of the Taubes conjecture in the fixed point free case is that the singular locus of the quotient orbifold $M$ is a fibered link and that 

%\noindent We remark that this corollary was known in the fixed point free case under the assumption of vanishing canonical class (\cite{Chen}, Proposition 1.8)
\section{Proof of Theorem \ref{class_eff}}
The proof of Theorem \ref{class_eff} is based on the following lemma, which provides a generalisation of (\cite{EL}, Theorem 5.2) to include irrational classes. For the proof we assume a certain familiarity with the basic properties of the Thurston norm (cf.\ \cite{Th}).
\begin{lem}\label{fin_action}
Let $\overline{M}$ be a 3-manifold with an orientation preserving smooth action of a finite group $G$ and quotient orbifold $M=\overline{M}/G$. An element $\overline{\phi}$ in the invariant subspace $H^1(\overline{M},\mathbb{R})^G$ admits a non-degenerate de Rham representative if and only if it admits a non-degenerate de Rham representative that is $G$-invariant. 

In particular, the restriction of the associated class $\phi \in H^1(|M|,\mathbb{R}) \cong H^1(\overline{M},\mathbb{R})^G$ to $M_{reg}$ lies in the open cone over a fibered face of the Thurston norm ball.

\end{lem}
\begin{proof}
We first assume that $\overline{\phi}$ is rational. Since nothing changes after multiplying with positive constants, we may assume that $\overline{\phi}$ is in fact integral. In this case the first claim is just a restatement of (\cite{EL}, Theorem 5.2), which can be applied in complete generality in view of (\cite{MS}, Theorem 8.1).  Note that the assumption $H^1(\overline{M},\mathbb{Q})^G = \mathbb{Q}$ in (\cite{EL}, Theorem 5.2) can be replaced by the fact that the fibration is given by a fibered class $\overline{\phi}$ that is $G$-invariant. Moreover, the proof in \cite{EL} actually gives a fibration that is transverse to the branching locus in $\overline{M}$. The quotient map $\pi$ induces an isomorphism $H^1(\overline{M},\mathbb{R})^G \cong H^1(|M|,\mathbb{R})$ so that there is a unique class $\phi$ with $\overline{\phi} = \pi^* \phi$ and the fibration dual to $\overline{\phi}$ descends to a fibration of $|M|$ dual to $\phi$. Finally since the fibration is transverse to the singular locus it follows that the restriction of $\phi$ to $M_{reg}$ is fibered.

We next assume that $\overline{\phi}$ is irrational and let $\phi$ be the unique class with $\overline{\phi} = \pi^* \phi$. We let $\iota_{reg}$ denote the natural inclusion $M_{reg} \hookrightarrow |M|$ and set $V = Im(\iota^*_{reg})$. By the previous case all rational classes in $V$ that are sufficiently close to $\phi_{reg} = \iota^*_{reg} \phi$ are fibered. If $\phi_{reg}$ itself did not lie in the open cone over a fibered face of the Thurston unit ball, then it must lie in the closed cone over the boundary of a fibered face by the assumption that it can be approximated by fibered elements. Since the Thurston unit ball is rational, the intersection of the closed cone containing $\phi_{reg}$ with $V$ must contain non-fibered rational points arbitrarily close to $\phi_{reg}$, which gives a contradiction. Thus $\phi_{reg}$ admits a non-degenerate de Rham representative $\eta_{reg}$. Since $\eta_{reg}$ can be approximated by rational classes that are fibered and restrict to meridional fibrations on the boundary of $M_{reg}$, the foliation induced by $\eta_{reg}$ on the boundary is also meridional.

We let $(z,\theta) \in D^2 \times S^1$ denote coordinates on a tubular neighbourhood of a component of the branching locus of $|M|$. After applying a suitable isotopy we may assume that $\eta_{reg}$ has the form $f(\theta)d \theta$ near $\partial D^2 \times S^1$. It follows that $\eta_{reg}$ extends to a non-degenerate closed form $\eta$ which is transverse to the branching locus of $|M|$. The pullback $\overline{\eta} = \pi^*\eta$ then gives the desired non-degenerate $G$-equivariant representative of $\overline{\phi}$.

%In particular, $\eta_{reg}$ restricts to a multiple of an integral fibered class on $\partial M_{reg}$.
\end{proof}
\begin{proof}[Proof of Theorem \ref{class_eff}]
The implication $(2)  \Longrightarrow (1)$ is trivial.

$(1)  \Longrightarrow (3)$: Let $(X,\omega)$ be a symplectic manifold with a fixed point free $S^1$-action and quotient space $M$. By (\cite{Chen}, Proposition 1.8) there is a manifold $\overline{M}$ and a smooth action by a finite group so that $M = \overline{M}/G$. Furthermore, we have the following commutative diagram:
\begin{equation*}
\xymatrix{ \pi^*X = \overline{X} \ar[r]^-{\overline{p}} \ar[d]_{\overline{\pi}} & \overline{M}  \ar[d]^{\pi}\\
X  \ar[r]^p & M, }
\end{equation*}
where $\pi$ is the quotient map, $\overline{\pi}$ is an unramified covering and the induced $S^1$-action on $\overline{X}$ is free. Moreover, the group $G$ acts naturally on $\overline{X}$ as the group of deck transformations of $\overline{\pi}$.

%Note that $p_*$ can be defined on $H^2(X,\mathbb{R})$ by forcing the above diagram to commute and using the isomorphism $H^1(\overline{M},\mathbb{R})^G \cong H^1(M,\mathbb{R})$.

Thus $\overline{\omega} = \overline{\pi}^*\omega$ is a symplectic form and by (\cite{FVi3}, Theorem 1.4) its image under the pushforward map $\overline{p}_*(\overline{\omega}) \in H^1(\overline{M},\mathbb{R})$ lies in the open cone over a fibered face of the Thurston norm ball. Since $\overline{\omega}$ is $G$-invariant and the action on $\overline{X}$ is fiber preserving, the class $\overline{\phi} = \overline{p}_*(\overline{\omega})$ is also $G$-invariant. We let $\phi$ be the unique class such that $\pi^*\phi = \overline{\phi}$. By Lemma \ref{fin_action} the restriction $\phi_{reg}$ to $M_{reg}$ lies in the open cone over a fibered face. Finally the naturality of the transfer homomorphism implies that the restriction of $\phi_{reg}$ agrees with $p_*\omega_{reg}$.

$(3) \Longrightarrow (2)$: By assumption $\phi_{reg} = p_*\psi_{reg}$ lies in the open cone over a fibered face of the Thurston norm ball and $\phi_{reg}$ is the restriction of a class $\phi \in H^1(|M|,\mathbb{R})$. In particular, $|M|$ fibers over $S^1$. We first note that $M$ is a \emph{very good} orbifold so that it is a quotient of a manifold $\overline{M}$ by a smooth action of a finite group $G$. For this it suffices to rule out bad 2-suborbifolds by (\cite{BLP}, Corollary 3.28). However, a bad 2-suborbifold is topologically a sphere that is essential in $H_2(|M|,\mathbb{Z})$ and as in the proof of (\cite{Chen}, Lemma 2.3) this implies that $b^+_2(X) = b_2(|M|) - 1$. Thus $|M| = S^2 \times S^1$ and $b^+_2(X) = 0$, contradicting the assumption that $\psi^2 >0$.

Thus since $M$ is very good we can proceed as in the proof of the previous implication. In particular, $M$ is a quotient of a manifold $\overline{M}$ by a smooth action of a finite group $G$, the total space has a finite covering $\overline{X}$ which is a genuine $S^1$-bundle and these bundles fit into a pullback diagram as above. Since a degree one cohomology class on $\overline{M}$ is determined by its restriction to the complement of the branching locus, we deduce that $\overline{\phi}= \pi^*\phi$ and $\overline{p}_*(\overline{\pi}^*\psi)$ agree as cohomology classes. We then note that the construction of $S^1$-invariant forms in \cite{FGM} and its extension to irrational classes (\cite{FVi2}, Theorem 1.1) can be done $G$-equivariantly.

First choose a $G$-invariant representative $\gamma$ of $e(\overline{X})$, which can be obtained as the curvature of a $G$-equivariant angular form. By (\cite{FVi2}, Lemma 2.1) we may write $\gamma = \overline{\phi} \wedge \beta$. After averaging over $G$ this equation still holds, so $\beta$ can be assumed to be $G$-equivariant. Let $\eta$ be a $G$-invariant angular 1-form so that $d\eta = \overline{p}^*\gamma$ and let $\overline{\Omega} \in H^2(\overline{M},\mathbb{R})$ be the unique class such that the following holds in cohomology
$$\overline{\pi}^*\psi - \eta  \wedge \overline{p}^* \thinspace \overline{\phi} = \overline{p}^*\overline{\Omega}.$$
Such an $\overline{\Omega}$ exists in view of the Gysin sequence since the left hand lies in the kernel of $\overline{p}_*$ and since the left hand side is $G$-equivariant so is $\overline{\Omega}$. The fact that $\overline{\pi}^*\psi^2 >0$ implies that $\overline{p}^* \thinspace \overline{\phi}  \wedge \overline{\Omega} > 0$. Thus by (\cite{FVi2}, Lemma 2.2) there is a non-vanishing 2-form representing the class $\overline{\Omega}$ so that $\overline{\phi} \wedge \overline{\Omega} > 0$, again after averaging we may assume that $\overline{\Omega}$ is $G$-invariant. Thus the $S^1$-invariant form
\[\overline{\omega}_{inv}= \eta \wedge \overline{p}^* \thinspace \overline{\phi} + \overline{p}^* \thinspace\overline{\Omega}\]
represents $\overline{\pi}^*\psi$ and descends to an $S^1$-invariant form $\omega_{inv}$ on $X$ which is cohomologous to $\psi$.

% and $\pi^*$ is an isomorphism onto $H^1(\overline{M},\mathbb{R})^G$
% so that $p_*(\omega_{inv}) = \phi \in H^1(M,\mathbb{Q}) \cong H^1(\overline{M},\mathbb{Q})^G$ for any fibered class $\phi$ on $M$. The cone of the fibered face determined by $\phi_r$ is just the intersection of the cone over the fibered face determined by $\overline{\phi}_r$ with $H^1(\overline{M},\mathbb{R})^G$, which then contains $\overline{\phi}$ by consruction.

%\medskip

%$(3) \Rightarrow (2)$: 

%and its extension for irrational classes by Friedl and Vidussi, \cite{FVi2}
\end{proof}
\begin{rem}\label{Chen_rem}
A vital step in the proof of Theorem \ref{class_eff} was Chen's observation that the base orbifold of a symplectic manifold $X$ admitting a fixed point free $S^1$-action is a quotient of a manifold by a finite group action. The main technical point in the proof of (\cite{Chen}, Proposition 1.8) is to rule out bad 2-orbifolds in the base. This is achieved by results relating the Seiberg-Witten invariants of the base orbifold to those of the underlying manifold. 

We sketch a different proof which uses more standard Seiberg-Witten vanishing results. For background on the Seiberg-Witten invariants we refer to \cite{GS} and the references therein. First observe that a bad 2-suborbifold $\Sigma$ in the quotient orbifold $M = X/S^1$ can intersect at most 2 singular curves $L_1,L_2$ each in at most one point. Taking a neighbourhood $N$ of $|\Sigma| \cup L_1$ gives a topological splitting of the base $|M| =  (S^2 \times S^1) \# M'$ so that preimage of the splitting sphere in $|M|$ induces a splitting $X=X_1 \cup_{S} X_2$, where $S$ is either $S^2 \times S^1$ or $S^3$ depending on whether $L_2$ is empty or not. Moreover, as in the proof of (\cite{Chen}, Lemma 2.3) we must have $b_1(M') > 0$ by the assumption that $b_2^+(X) > 0$. If $S$ is a 3-sphere, then $b_2^+(X_2) > 0$ and by taking the covering $\overline{X}$ of $X$ induced by the natural surjection
\[\pi_1(X) \to \pi_1(S^2 \times S^1) \to \mathbb{Z}_n\]
we obtain a splitting of $\overline{X} = \overline{X}_1 \cup_{S^3} \overline{X}_2$, where $b_2^+(\overline{X}_1),b_2^+(\overline{X}_2) \geq 1$. It follows that the Seiberg-Witten invariants of $\overline{X}$ are trivial.

If $S = S^2 \times S^1$, then we take the covering $\overline{X}$ of $X$ induced by a surjection
\[\pi_1(X) \to \pi_1(S^2 \times S^1) * \pi_1(|M'|) \to \mathbb{Z}_n \times \mathbb{Z}_n.\]
The embedded $2$-sphere $S^2 \times \{pt\}$ in $S$ then becomes essential in the covering and $b_1(\overline{X})$, and hence $b^+_2(\overline{X})$, may be assumed to be arbitrarily large. Furthermore, the sphere $S^2 \times \{pt\}$ has trivial self-intersection and consequently the Seiberg-Witten invariants of $\overline{X}$ are trivial. Thus in both cases we obtain a contradiction to the non-vanishing results of Taubes for the Seiberg-Witten invariants of a symplectic 4-manifold.
\end{rem}

\section*{Acknowledgments:}
\noindent We thank P. Su\'{a}rez-Serrato for his stimulating questions and S. Friedl for helpful comments. The hospitality of the Max Planck Institute f\"ur Mathematik in Bonn, where this research was carried out, is also gratefully acknowledged.

 \end{document}